\documentclass[11pt,french, english]{amsart}
\usepackage[T1]{fontenc}
\usepackage[latin1]{inputenc}
\usepackage{latexsym}
\usepackage{textcomp}
\usepackage{amsthm}
\usepackage{amstext}
\usepackage{graphicx}
\usepackage{amssymb}
\usepackage[all]{xy}
\usepackage[percent]{overpic}

\makeatletter

\DeclareRobustCommand{\greektext}{%
  \fontencoding{LGR}\selectfont\def\encodingdefault{LGR}}
\DeclareRobustCommand{\textgreek}[1]{\leavevmode{\greektext #1}}
\DeclareFontEncoding{LGR}{}{}
\DeclareTextSymbol{\~}{LGR}{126}
\newcommand{\lyxmathsym}[1]{\ifmmode\begingroup\def\b@ld{bold}
  \text{\ifx\math@version\b@ld\bfseries\fi#1}\endgroup\else#1\fi}

\providecommand{\tabularnewline}{\\}

\numberwithin{equation}{section}
\numberwithin{figure}{section}
\theoremstyle{plain}
\newtheorem{thm}{Theorem}
  \theoremstyle{plain}
  \newtheorem{conjecture}[thm]{Conjecture}
  \theoremstyle{plain}
  \newtheorem{prop}[thm]{Proposition}
  \theoremstyle{remark}
  \newtheorem{rem}[thm]{Remark}
  \theoremstyle{plain}
  \newtheorem{cor}[thm]{Corollary}
 \theoremstyle{definition}
  \newtheorem{example}[thm]{Example}
  \theoremstyle{definition}
  \newtheorem{problem}[thm]{Problem}
  \theoremstyle{plain}
  \newtheorem{lem}[thm]{Lemma}
  \theoremstyle{definition}
  \newtheorem{defn}[thm]{Definition}

\AtBeginDocument{
  
}

\makeatother

\usepackage{babel}
\addto\extrasfrench{\providecommand{\fg}{\ifdim\lastskip>\z@\unskip\fi~\frqq}}

\begin{document}

 \def\cA{{\mathcal A}} \def\cB{{\mathcal B}} \def\cC{{\mathcal C}} \def\cD{{\mathcal D}} \def\cE{{\mathcal E}} \def\cF{{\mathcal F}} \def\cG{{\mathcal G}} \def\cH{{\mathcal H}} \def\cI{{\mathcal I}} \def\cJ{{\mathcal J}} \def\cK{{\mathcal K}} \def\cL{{\mathcal L}} \def\cM{{\mathcal M}} \def\cO{{\mathcal O}} \def\cP{{\mathcal P}} \def\cQ{{\mathcal Q}} \def\cR{{\mathcal R}} \def\cS{{\mathcal S}} \def\cU{{\mathcal U}} \def\cV{{\mathcal V}} \def\cW{{\mathcal W}} \def\cX{{\mathcal X}} \def\cY{{\mathcal Y}} 
\def\bB{{\mathbb B}} \def\bC{{\mathbb C}} \def\bD{{\mathbb D}} \def\bF{{\mathbb F}} \def\bR{{\mathbb R}} \def\bP{{\mathbb P}} \def\bN{{\mathbb N}} \def\bQ{{\mathbb Q}} 
\def\A{{\bf A}}  \def\B{{\bf B}}  \def\Z{{\bf Z}}
\def\N{{\bf N}}  \def\P{{\bf P}} \def\F{{\bf F}}
\def\M{{\bf M}}  \def\R{{\bf R}}  \def\Q{{\bf Q}}
\def\C{{\bf C}}

\title{On the hyperbolicity of surfaces of general type with small $c_1 ^2$}

\author{Xavier Roulleau, Erwan Rousseau}
\thanks{The first author is supported by grant FCT SFRH/BPD/72719/2010 and project Geometria Alg\'ebrica PTDC/MAT/099275/2008. The second author is supported by the Agence Nationale de la Recherche (ANR) through projects POSITIVE (ANR-10-BLAN-0119) and COMPLEXE (ANR-08-JCJC-0130-01).}

\maketitle

\begin{abstract}
Surfaces of general type with positive second Segre number $s_2:=c_1^2-c_2>0$ are known by results of Bogomolov to be algebraically quasi-hyperbolic i.e. with finitely many rational and elliptic curves. These results were extended by McQuillan in his proof of the Green-Griffiths conjecture for entire curves on such surfaces. 
In this work, we study hyperbolic properties of minimal surfaces of general type with minimal $c_1^2$, known as Horikawa surfaces. In principle these surfaces should be the most difficult case for the above conjecture as illustrate the quintic surfaces in $\bP^3$. Using orbifold techniques, we exhibit infinitely many irreducible components of the moduli of Horikawa surfaces whose very generic member has no rational curves or even is algebraically hyperbolic. Moreover, we construct explicit examples of algebraically hyperbolic and quasi-hyperbolic orbifold Horikawa surfaces.
\end{abstract}


\section{Introduction}

Our motivations are the following conjectures of Green-Griffiths and
Lang:
\begin{conjecture}[Lang \cite{Lang}]
\label{Lang} Let $X$ be a complex variety of general
type. Then a proper Zariski closed subset $Z$ of $X$ contains all
its subvarieties not of general type. In particular, $X$ has only
a finite number of codimension-one subvarieties not of general type. 
\end{conjecture}
Even in the case of surfaces, this conjecture is still open, except in some specific cases such as surfaces of general type with irregularity at least two \cite{Lu}. It has attracted a lot of attention because of its important conjectural links with arithmetic: according to the Bombieri-Lang conjecture, the rational points on a variety of general type defined over a number field should not be Zariski dense.

A surface satisfying Conjecture \ref{Lang}, i.e. with finitely many rational and elliptic curves, is said to be \emph{algebraically quasi-hyperbolic}. Following the terminology of Demailly \cite{De95}, a projective variety $X\subset\mathbb{P}^{n}$ is called  \emph{algebraically
hyperbolic} if there exists a positive real number \textgreek{e} such
that \[
2g(C)\lyxmathsym{\textminus}2\ge\lyxmathsym{\textgreek{e}}\deg C\]
for each reduced irreducible curve $C\subset X$, where $g(C)$ and
$\deg C$ are the geometric genus and the degree of the curve $C\hookrightarrow X$
respectively. 

 The analytic version of Conjecture \ref{Lang} is:
  
\begin{conjecture}[Green-Griffiths \cite{GG80}, Lang \cite{Lang}]
\label{GGL}
Let $X$ be a variety of general type. Then a proper Zariski closed
subset of $X$ contains all images of entire curves $f:\bC\to X$. 
\end{conjecture}

A projective variety satisfying Conjecture \ref{GGL}  is said to be \emph{quasi-hyperbolic} and a variety that contains no non-constant entire curve is said  to be \emph{hyperbolic}.

Let us recall that minimal surfaces of general type $X$ are classified according to their Chern numbers. These  Chern numbers $c_1^2,c_2$ satisfy the two famous inequalities of Noether and Bogomolov-Miyaoka-Yau (see for example \cite{Barth}): 
$$\frac{1}{5}(c_2- 36) \leq c_1^2 \leq 3 c_2.$$

If $c_1^2=3c_2$ then, by Yau \cite{Y1}, $X$ is a quotient of the unit ball. Therefore $X$ is hyperbolic.

A striking overhang toward Conjecture \ref{Lang} is
its proof by Bogomolov for surfaces whose Segre number $s_{2}=c_{1}^{2}-c_{2}$
is positive \cite{Bog77}. The key point in Bogomolov's proof is that for surfaces
with positive Segre number, the sufficiently high symmetric differentials
have global sections.

An extension of Bogomolov's result to its analytic version was obtained by Lu and Yau \cite{LuYa} proving conjecture \ref{GGL} for surfaces with $c_1^2 >2c_2$.

Then McQuillan \cite{McQ0} proved Conjecture \ref{GGL} for surfaces with positive second Segre number.

For surfaces with $s_{2}\leq0$ there is no such a good
result. Demailly and El Goul \cite{DeEG} proved Conjecture \ref{GGL} for some surfaces with $13c_1^2 >9c_2$.

In the extreme case, a surface that reaches the equality $c_{2}=5c_{1}^{2}+36$ if $c_{1}^{2}$
is even and $c_{2}=5c_{1}^{2}+30$ otherwise is called a Horikawa surface.
Thus the Horikawa surfaces are the one for which the Segre number
$s_{2}=c_{1}^{2}-c_{2}$ is the most negative as possible in the geography
of (minimal) surfaces of general type.

As the previous list of results suggests, one is naturally lead to believe that it is for these surfaces that the above conjectures will be the most difficult to establish.

Even for hypersurfaces in $\mathbb{P}^{3}$, (which have $s_2<0$), we are far from complete results.
The case of quintics ($c_1 ^2 =5, c_2 =55$) is particularly difficult to treat and this may be explained by the fact that they are Horikawa surfaces, thus with minimal second Segre number. Indeed, a conjecture by Kobayashi predicts that generic quintics should be hyperbolic. 
We know by Geng Xu \cite{Xu} that a very generic surface of degree
$d\geq5$ in $\bP^{3}$ contains no curve of geometric genus $0$
or $1$. But we still do not know that they are hyperbolic, even algebraically hyperbolic.

Moreover, we have very few examples of hyperbolic surfaces
of low degree. A striking fact is that no example of hyperbolic surface is known for $d=5$. 

In this paper we investigate the hyperbolic properties of Horikawa surfaces using their very interesting geometric properties. Indeed, we know that most Horikawa surfaces realizing the equality $c_{2}=5c_{1}^{2}+36$
appear as ramified coverings \cite{Hor76}. Therefore it is very natural to associate to such surface an orbifold.
Our philosophy is to use this orbifold to study the geometry of curves in Horikawa surfaces.

We will see that ''orbifold'' techniques as systematically introduced
by Campana (\cite{Campana04}, \cite{Campana07}, see also \cite{Rousseau})
are useful in this context.

Let $N \geq 0$ be an integer and let $\mathbb{F}_N$ be the $N^{th}$ Hirzebruch surface. The surface $\mathbb{F}_N$ has a natural fibration $\mathbb{F}_N \to \mathbb{P}^1$. We denote by $F$  a fiber, and by $T$ the section of this fibration such that  $T^{2}=N$. Any divisor $D$ on $\mathbb{F}_N$ is numerically equivalent to $aT+bF$ for $a,b$ integers. We denote by $(a,b)$ the equivalence class of $D$. The holomorphic Euler characteristic of a surface is $\chi:=\frac{c_{1}^{2}+c_{2}}{12}$. Horikawa obtained the following classification \cite{Hor76}: a Horikawa surface with even $c_1^2$ is either
\begin{enumerate}
\item a double covering of $\bP^2$ branched along an octic ($\chi=4$),
\item a double covering of $\bP^2$ branched along a curve of degree $10$ ($\chi=7$),
\item a double covering of $\bF_N$ branched along a curve of type $(6,2a)$, ($2a\geq -N$) ($\chi=3N+2a-1$).
\end{enumerate}
Here the branch curve has at most $ADE$ singularities. We show:

\begin{thm}
\label{thmprincipal}
\begin{enumerate}
\item Let $X$ be a very generic Horikawa surface of type $(1)$. Then $X$ has no rational curves.
\item Let $X$ be a very generic Horikawa surface of type $(2)$. Then $X$ is algebraically hyperbolic, in particular $X$ has no rational or elliptic curves.
\item Let $X$ be a very generic Horikawa surface of type $(3)$ with $a=3$ and $\chi=3N+5$. Then $X$ has no rational curves.
\end{enumerate}
\end{thm}

Here, the terminology ``very generic'' is used to indicate that the property is satisfied outside a countable union of algebraic subsets in the moduli space.
This result is obtained as a consequence of more general results on the algebraic hyperbolicity of branched covers which may be of independent interest.

One of the most natural examples of surfaces with $s_{2}<0$ are hypersurfaces
in $\mathbb{P}^{3}$. In \cite{Bogomolov2}, Bogomolov and De Oliveira
studied the hyperbolicity of hypersurfaces of degree $d>5$ with sufficiently many nodes. Translated into the langue of orbifold, their key observation is that for such a surface $X$, the natural structure of orbifold $\mathcal{X}$
has positive orbifold Segre number $s_{2}(\mathcal{X})>0$. We investigate hyperbolic properties of the quintic surfaces that are degree $5$ cyclic covers of $\bP^2$ branched over $5$ lines in general position. In this study, we develop the theory of jet differentials in the orbifold setting and, as an application, prove

\begin{thm}
\label{Quint}
Let $X$ be a quintic surface that is a  degree $5$ cyclic cover of $\bP^2$ branched over $5$ lines in general position. Then any entire orbifold curve $f: \bC \to X$ satisfies a differential equation of order $2$.
\end{thm}

Using Nevanlinna theory and constructions of Persson \cite{Persson} of Horikawa surfaces with maximal Picard number, we exhibit some explicit
examples of quasi-hyperbolic Horikawa orbifold surfaces. 
\begin{thm}
For $\chi$ equal to $4$ and $2k-1$ (for any integer $k>2$), there exists  quasi-hyperbolic orbifold Horikawa surfaces
whose minimal resolutions have Euler characteristic $\chi$.
\end{thm}

To finish this introduction, let us remark that  in characteristic $p>0$ the analog of the Lang conjecture is false. In \cite{Liedtke}, unirational Horikawa surfaces are constructed over fields of positive characteristic, these unirational surfaces have maximal Picard number.

The paper is structured as follows. In section $2$ we recall the definition of orbifolds and the classical results we will use. In section $3$ we prove our main Theorem \ref{thmprincipal} and some further results on surfaces that are cyclic ramified covers of the plane. In section $4$ we prove Theorem \ref{Quint}. Section $5$ is devoted to the construction of quasi-hyperbolic Horikawa orbifold surfaces.

\textbf{Acknowledgements}. Part of this research was done during the authors stay in Strasbourg University. We thank Matthias Sch\"utt for many discussions on singularities of quintic surfaces, Julien Grivaux for discussions on orbifolds, Beno\^it Claudon and Jevgenija Pavlova for their drawings. We also thank the referee for his remarks and suggestions.

\section{Orbifold set-up}

\subsection{Orbifolds}

As in \cite{GK}, we define orbifolds as a particular type of log pairs. The data $(X,\Delta)$ is a log pair if $X$ is a normal algebraic variety (or a normal complex space) and $\Delta=\sum_id_iD_i$ is an effective $\bQ$-divisor where the $D_i$ are distinct, irreducible divisors and $d_i \in \bQ$. 

For orbifolds, we need to consider only pairs $(X,\Delta)$ such that $\Delta$ has the form
$$\Delta=\sum_i \left(1-\frac{1}{m_i}\right) D_i,$$
where the $D_i$ are prime divisors and $m_i \in \bN$. 

\begin{defn}
An {\em orbifold chart} on $X$ compatible with $\Delta$ is a Galois covering $\varphi : U \to \phi(U)\subset X$ such that
\begin{enumerate}
\item $U$ is a domain in $\bC^n$ and $\varphi(U)$ is open in $X$,
\item the branch locus of $\varphi$ is $\lceil \Delta \rceil \cap \varphi(U)$,
\item for any $x \in U'':=U\setminus \varphi^{-1}(X_{sing}\cup \Delta_{sing})$ such that $\varphi(x) \in D_i$, the ramification order of $\varphi$ at $x$ verifies $ord_{\varphi}(x)=m_i.$
\end{enumerate}
\end{defn}

\begin{defn}
An orbifold $\mathcal{X}$ is a log pair $(X,\Delta)$ such that $X$ is covered by orbifold charts compatible with $\Delta.$
\end{defn}

\begin{rem}
\begin{enumerate}
\item In the language of stacks, we have a smooth Deligne-Mumford stack $\pi: \mathcal{X} \rightarrow X$, with coarse moduli space $X$.
\item More generally, Campana introduces {\em geometric orbifolds} in \cite{Campana04} as pairs of this type but which are not necessarily locally uniformizable.
\end{enumerate}
\end{rem}

\begin{example}\label{nc}
Let $X$ be a complex manifold and $\Delta=\sum_{i} (1-\frac{1}{m_{i}})D_{i}$ with a support $\lceil \Delta \rceil$ which is a normal crossing divisor, i.e. for any point $x\in X$ there is a holomorphic coordinate system $(V,z_1,\dots,z_n)$ such that $\Delta$ has equation $$z_{1}^{(1-\frac{1}{m_{1}})}\dots z_{n}^{(1-\frac{1}{m_{n}})}=0.$$ Then $(X,\Delta)$ is an orbifold. Indeed, fix a coordinate system as above. Set
$$\varphi: U \to V, \text{   } \phi(x_1,\dots,x_n)=(x_1^{m_1},\dots,x_n^{m_n}).$$
Then $(U,\phi)$ is an orbifold chart on $X$ compatible with $\Delta$.
\end{example}

In the case of surfaces we have more examples of orbifolds looking at singularities that naturally appear in the logarithmic Mori program.
\begin{defn}
Let $(X,\Delta)$, $\Delta=\sum_i\left(1-\frac{1}{m_i}\right)C_i$, be a pair where $X$ is a normal surface and $K_X+\Delta$ is $\mathbb{Q}$-Cartier. Let $\pi: \tilde{X} \to X$ be a resolution of the singularities of $(X,\Delta)$, so that the exceptional divisors, $E_i$ and the components of $\tilde{\Delta}$, the strict transform of $\Delta$, have normal crossings and
$$K_{\tilde{X}}+\tilde{\Delta}+\sum_iE_i=\pi^*(K_X+\Delta)+\sum_ia_iE_i.$$

We say that $(X,\Delta)$ is \emph{klt} (Kawamata log terminal) if $m_i<\infty$ and $a_i > 0$ for every exceptional curve $E_i$.

\end{defn}

Thanks to a result of \cite{Nak89} (see also the appendix of \cite{Campana10}), we have
\begin{example}\label{klt}
Let $(X,\Delta)$, $\Delta=\sum_i\left(1-\frac{1}{m_i}\right)C_i$, be a klt pair with $X$ a surface, then $(X,\Delta)$ is an orbifold.
\end{example}

\subsection{Chern classes}

\label{SubsectionChernClasses}

Let $\pi: \mathcal{X} \to (X, \Delta)$ be a two dimensional orbifold. Let $S$ be the singular
points of $X$ and of the divisor $\left\lceil \Delta\right\rceil$, $\Delta=\sum(1-\frac{1}{m_{i}})D{}_{i}$. The orbifold canonical line bundle is:
$K_\mathcal{X}=\pi^*(K_X+\Delta),$ therefore:
$$c_{1}(\mathcal{X})=-\pi^{*}(K_{X}+\sum(1-\frac{1}{m_{i}})D{}_{i}).$$

To each
point $p$ of $S$, there is a well defined integer $\beta(p)$, the order of the isotropy group,
such that by the orbifold Gauss-Bonnet formula \cite{Toen}:
 
$$c_{2}(\mathcal{X})=e(X)-\sum(1-\frac{1}{m_{i}})e(D_{i}\setminus S)-\sum_{p\in S}(1-\frac{1}{\beta(p)}).$$

Let $p\in S$ be a smooth point of $X$.  

\begin{lem}
\label{Beta} 
Suppose that $p$ is an $ADE$ singularity of $\left\lceil \Delta\right\rceil$. We have :
 
\begin{tabular}{|c|c|}
\hline 
Type & $\beta(p)$\tabularnewline
\hline 
$A_{1}$ & $m_{i}m_{j}$\tabularnewline
\hline 
$A_{2n}$ & $\frac{2}{2n+1}(\frac{1}{m_{i}}+\frac{1}{2n+1}-\frac{1}{2})^{-2}$\tabularnewline
\hline 
$A_{2n-1}$ & $\frac{4}{n}(\frac{1}{m_{i}}+\frac{1}{m_{j}}+\frac{1}{n}-1)^{-2}$,
$n\geq2$\tabularnewline
\hline 
$D_{2n+2}$ & $\frac{4}{n}(\frac{1}{m_{i}}+\frac{1}{m_{j}}+\frac{1}{nm_{k}}-1)^{-2}$\tabularnewline
\hline 
$D_{2n+3}$ & $2(2n+1)m_{i}^{2}$, $n\geq1$\tabularnewline
\hline 
$E_{7}$ & $96$\tabularnewline
\hline
\end{tabular} 

where the branches are as follows:
\begin{figure}[!h]
\setlength{\unitlength}{1.2mm}
\begin{picture}(120,20)
\qbezier(1,1)(11,8)(1,15)\qbezier(11,1)(1,8)(11,15)\put(6,8){\circle*{1}}
\put(7,9){$$}\put(-1,9){$ $}\put(-1,16){$m_i$}\put(10,16){$m_j$}

\qbezier(24,1)(31,6)(31,15)\qbezier(31,15)(31,6)(38,1)\put(31,15){\circle*{1}}
\put(28,17){$$}\put(27,1){$m_i$}
\qbezier(48,1)(58,8)(48,15)\qbezier(58,1)(48,8)(58,15)\put(53,8){\circle*{1}}\put(48,8){\line(1,0){10}}
\put(54,9){$$}\put(46,9){$m_k$}\put(46,16){$m_i$}\put(57,16){$m_j$}
\qbezier(68,1)(75,6)(75,15)\qbezier(75,15)(75,6)(82,1)\put(75,15){\circle*{1}}\put(68,15){\line(1,0){14}}
\put(72,17){$ $}\put(83,15){$m_i$}\put(71,1){$2$}
\qbezier(92,1)(99,6)(99,15)\qbezier(99,15)(99,6)(106,1)\put(99,15){\circle*{1}}\put(99,1){\line(0,1){14}}
\put(96,17){$ $}\put(100,1){$2$}\put(105,3){$2$}
\end{picture}

\begin{picture}(120,5)
\put(6,1){$A_{2n+1} $}
\put(27,1){$A_{2n} $}
\put(49,1){$D_{2n+2}$}
\put(72,1){$D_{2n+1}$}
\put(97,1){$E_{7}$}
\end{picture}
\caption{ADE singularities and conditions on the multiplicities}
\end{figure}
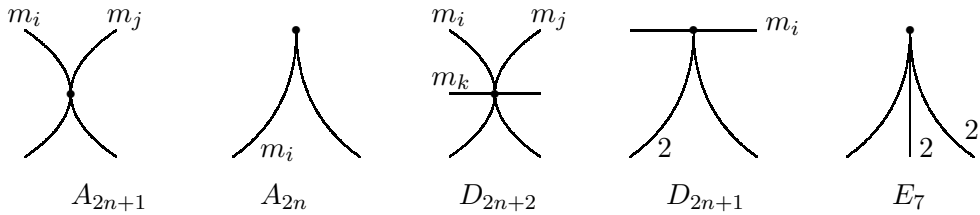

\end{lem}

\begin{proof}
See \cite{Uludag2} Table 2.3, \cite{Uludag}, \cite{Campana10} and
\cite{Kobayashi}.
\end{proof}
Let $C_{1},\dots$ be disjoint reduced divisors on $X$ whose irreducible
components are $(-2)$-curves. There exist a surface $X'$ and map
$X\rightarrow X'$ such that the $C_{i}$'s are contracted onto
$ADE$-singularities and that is an isomorphism outside. Let $a_{n}$
(resp $d_{n}$, $e_{n}$) be the number of $A_{n}$ (resp. $D_{n},\, E_{n}$)
singularities on $X'$. 

\begin{prop}
\label{ChernNumberOrbi}
(\cite{Megyesi}, 1.8). The surface $X'$ has a natural structure of orbifold $\mathcal{X}$
and its Chern numbers are $c_{1}^{2}(\mathcal{X})=c_{1}^{2}(X)$ and:

\begin{equation}
\label{Chernnumb}
c_{2}(\mathcal{X})=c_{2}(X)-\sum(n+1)(a_{n}+d_{n}+e_{n})+\sum\frac{a_{n}}{n+1}+\frac{d_{n}}{4(n-2)}+\frac{e_{6}}{24}+\frac{e_{7}}{48}+\frac{e_{8}}{120}.
\end{equation}

\end{prop}

The denominators $4(n-2)$, $24,\,48,\,120$ are the order of the
binary dihedral $BD_{4(n\text{\textminus}2)}$, the binary tetrahedral,
the binary octahedral, and the binary icosahedral group respectively.

\subsection{Orbifold Riemann-Roch}

Let $L$ be an orbifold line bundle on the orbifold $\mathcal{X}$ of dimension $n$. We will use Kawazaki's orbifold Riemann-Roch theorem \cite{Kawa} or To\"en's for Deligne-Mumford stacks \cite{Toen} using intersection theory on stacks.

\begin{thm}[\cite{Toen}]
Let $\mathcal{X}$ be a Deligne-Mumford stack with quasi-projective coarse moduli space and which has the resolution property (i.e every coherent sheaf is a quotient of a vector bundle). Let $E$ be a coherent sheaf on $\mathcal{X}$ then
$$\chi(\mathcal{X},E)=\int_{\mathcal{X}}\widetilde{ch}(E)\widetilde{Td}(T_{\mathcal{X}}).$$
\end{thm}

From this formula, we obtain the asymptotic:
$$\chi(\mathcal{X}, L^k)=\frac{c_1(L)^n}{n!}k^n+O(k^{n-1}),$$
using orbifold Chern classes.

It should be remarked that we use here the fact that we work only with effective orbifolds: the stabilizers are generically trivial therefore the correction terms of the orbifold Riemann-Roch do not affect the leading term.

We will apply this result to the case of orbifold surfaces $\mathcal{X}$ of general type with big orbifold cotangent bundle $\Omega_{\mathcal{X}}$. Indeed, let $T_\mathcal{X}$ be the orbifold tangent bundle. Then $\bP(T_\mathcal{X})$ is naturally an orbifold and we can apply the previous Riemann-Roch formula to the tautological line bundle $\mathcal{O}_{\bP(T_\mathcal{X})}(1)$. Orbifold Serre duality gives $$h^2(\mathcal{X}, S^m\Omega_{\mathcal{X}})=h^0(\mathcal{X}, S^m\Omega_{\mathcal{X}} \otimes  K_{\mathcal{X}}^{1-m}) \hookrightarrow h^0(\mathcal{X}, S^m\Omega_{\mathcal{X}}).$$

As a corollary, if $c_1^2(\mathcal{X})-c_2(\mathcal{X})>0$ then  $$H^0(S^m\Omega_{\mathcal{X}}) \geq c.m^3$$ for a positive constant $c$.

\subsection{Orbifold hyperbolicity}
In the context of manifolds, complex hyperbolicity is concerned with the geometry of rational, elliptic curves and more generally entire curves i.e. holomorphic maps from $\bC$. For orbifold hyperbolicity, algebraic curves are replaced by orbifold Riemann surfaces $\mathcal{C}$ i.e. the data of a Riemann surface $C$ of genus $g\geq 0$ and $r$ points $p_1, \dots, p_r \in X$ marked by orders of stabilizer groups $m_1,\dots,m_r \geq 2$ corresponding to the $\bQ$-divisor $\Delta= \sum_i \left(1-\frac{1}{m_i}\right)p_i$.

\begin{defn}
An orbifold Riemann surface $\mathcal{C}$ is \emph{ rational} (resp. \emph{elliptic}) if \\
$\deg(K_\mathcal{C})<0$ i.e. $2g-2+\sum_i \left(1-\frac{1}{m_i}\right) <0$ (resp. $\deg(K_\mathcal{C})=0$).
\end{defn}

\begin{example}
A case-by-case check gives that $(C,\Delta)$ is elliptic if it is isomorphic to one of the following orbifold curves:
\begin{itemize}
\item $(E,\emptyset)$ where $E$ is an elliptic curve.
\item $(\bP^1, \left(1-\frac{1}{m_1} \right)\{0\}+\left(1-\frac{1}{m_2} \right)\{1\}+\left(1-\frac{1}{m_3} \right)\{\infty\}$ where $(m_1, m_2, m_3)$ is either $(2, 3, 6), (2, 4, 4), (3, 3, 3)$.
\item $(\bP^1, \left(1-\frac{1}{2} \right)\{0\}+\left(1-\frac{1}{2} \right)\{1\}+\left(1-\frac{1}{2} \right)\{p_3\}+\left(1-\frac{1}{2} \right)\{\infty\}$ with $p_3 \in \bC \setminus \{0, 1\}$.
\end{itemize}
\end{example}

Recall that an \emph{orbifold  map} between orbifolds $f: \mathcal{X}_1 \to \mathcal{X}_2$ is a map between the underlying spaces $X_1$ and $X_2$ which lifts to an equivariant map in orbifold charts.

\begin{defn}
An orbifold map $p: \mathcal{X}_1 \to \mathcal{X}_2$ is an \emph{orbifold covering} if every $x \in X_2$ is in some $U \subset X_2$ such that for every component $V$ of $p^{-1}(U)$ the corresponding orbifold chart $\phi_1: \tilde{V} \to V$ verifies that $p\circ \phi_1: \tilde{V} \to U$ is an orbifold chart for $U$.
\end{defn}

\begin{defn}
Let $\mathcal{X}$ be an orbifold. An orbifold rational (resp. elliptic) curve in  $\mathcal{X}$ is the image of an orbifold map $f: \mathcal{C} \to \mathcal{X}$ where $\mathcal{C}$ is rational (resp. elliptic).
\end{defn}

A hyperbolic orbifold should not have rational and elliptic curves and more generally:

\begin{defn}
An orbifold $\mathcal{X}$ is \emph{(Brody-)hyperbolic} if there is no non-constant orbifold map $f: \bC \to \mathcal{X}$.
An orbifold $\mathcal{X}$ is \emph{quasi-hyperbolic} if the Zariski closure of the union of the images of orbifold maps $f: \bC \to \mathcal{X}$ is a proper sub-variety.
\end{defn}

\begin{example}
From the uniformization theorem available in the orbifold setting (see for example \cite{Far}), we have that an orbifold curve $\mathcal{C}$ is hyperbolic if and only if $\deg(K_\mathcal{C})>0$.
\end{example}

In the algebraic setting, we can generalize the notion of algebraic hyperbolicity to orbifolds:
\begin{defn}
A compact orbifold $\mathcal{X}$, with $\omega$ an orbifold hermitian metric, is \emph{algebraically hyperbolic} if there exists $\epsilon>0$ such that for any orbifold morphism $f: \mathcal{C} \to \mathcal{X}$ with $\mathcal{C}$ compact,
$$\deg(K_\mathcal{C})\geq \epsilon\int_C f^*\omega.$$
\end{defn}

We suggest the following generalization of Green-Griffiths-Lang conjecture in the orbifold context

\begin{conjecture}[Orbifold Green-Griffiths-Lang conjecture]\label{OGG}
Let $\mathcal{X}$ be an orbifold of general type. Then there exists a proper sub-orbifold $\mathcal{Z} \subset \mathcal{X}$ which contains the image of all holomorphic maps $f: \bC \to \mathcal{X}$.
\end{conjecture}

Even in the setting of manifolds, this conjecture is widely open. An important result was obtained by McQuillan \cite{McQ0} with the confirmation of the conjecture for surfaces of general type with $c_1^2-c_2 >0$. The positivity of the second Segre number ensures by a Riemann-Roch computation as already explained above that $H^0(S^m\Omega_X)>cm^3$. McQuillan has extended his result to $2$-dimensional Deligne-Mumford stacks with projective moduli \cite{McQ1} (see also \cite{Rousseau} and \cite{Rousseau2}).

Since orbifolds are examples of DM stacks, a consequence of his result that we shall use is
\begin{thm}\label{segre}
Let $\mathcal{X}$ be an orbifold surface of general type with positive orbifold second Segre number $$c_1^2(\mathcal{X})-c_2(\mathcal{X})>0.$$ Then there exists a proper sub-orbifold $\mathcal{Z} \subset \mathcal{X}$ which contains the image of all orbifold maps $f: \bC \to \mathcal{X}$, in other terms: the orbifold $\mathcal{X}$ is quasi-hyperbolic.
\end{thm}

Let $\mathcal{C}\to \mathcal{X}$ be an orbifold rational or elliptic curve. The subjacent space to $\mathcal{C}$ is $\mathbb {P}^1$ or a smooth elliptic curve, thus such a curve gives rise to an entire orbifold curve $\bC \to \cX$. Therefore a quasi-hyperbolic orbifold is algebraically quasi-hyperbolic in the sense that it has only finitely many rational and elliptic orbifold curves.

\begin{example}
Following \cite{Bogomolov2} (see also \cite{Rousseau}), we see that nodal surfaces in $\bP^3$ of degree $d$ with $l$ nodes provide examples of orbifolds with positive orbifold second Segre number if
$$l> \frac{8}{3}\left(d^2-\frac{5}{2}d \right).$$ It should be remarked that this numerical condition can be satisfied only for $d \geq 6$. Moreover, in \cite{Bogomolov2}, the authors claim that a subspace $H^0 (\mathcal{X},S^m \Omega_{\mathcal{X}})$ of symmetric orbifold differentials can be extended to the resolutions of such surfaces but it seems to us that there is a gap in the argument of \cite{Bogomolov2} Lemma 2.2.
\end{example}

\subsection{Orbifold jet differentials}
Since Bloch \cite{Blo}, jet differentials have turned out to be a useful tool to study hyperbolic properties of complex manifolds. We would like to use these technics in the orbifold setting extending the construction of Demailly \cite{De95} for manifolds to this situation.

Start with an orbifold $\mathcal{X}$ and an orbifold subbundle $\mathcal{V} \subset T_\mathcal{X}$. Then we consider the orbifold $\tilde{\mathcal{X}}:=\bP(\mathcal{V})$ with its natural orbifold line bundle
$$\mathcal{O}_{\tilde{\mathcal{X}}}(-1).$$ We have an orbifold map $\pi: \tilde{\mathcal{X}} \to \mathcal{X}.$ We define $$\tilde{\mathcal{V}}:=\pi_*^{-1}\mathcal{O}_{\tilde{\mathcal{X}}}(-1) \subset T_{\tilde{\mathcal{X}}},$$
here $\pi_*$ is the differential map between $T_{\tilde{\mathcal{X}}}$  and $\pi^* T_{\mathcal{X}}$ (see \cite{DR}).

Therefore, as in the case of manifolds, we have an inductive process which gives the orbifold Demailly-Semple jet bundles:
$$(\mathcal{X}_0,\mathcal{V}_0)=(\mathcal{X},\mathcal{V}), (\mathcal{X}_k,\mathcal{V}_k)=(\tilde{\mathcal{X}}_{k-1},\tilde{\mathcal{V}}_{k-1}).$$

Taking $\mathcal{V}_0:=T_{\mathcal{X}_0}$, one can define jet differentials $E_{k,m}$ as the direct image sheaf
$$E_{k,m}:= (\pi_{0,k})_*\mathcal{O}_{{\mathcal{X}_k}}(m),$$
where $\pi_{0,k}:= \mathcal{X}_k \to \mathcal{X}$ is the natural projection map.

\begin{rem}
Alternatively, following \cite{De95} one could construct Green-Griffiths jet differentials $E_{k,m}^{GG}$ on orbifolds and define $E_{k,m}$ as the subbundle of orbifold jet differentials which are invariant under the reparametrization of germs of orbifold curves.
\end{rem}

One can use the orbifold Riemann-Roch to compute the Euler characteristic. In the case of orbifold surfaces, one obtains
$$\chi(\mathcal{X}, E_{k,m})=m^{k+2}(\alpha_k c_1^2 - \beta_k c_2)+ O(m^{k+1}),$$
where $c_1$ and $c_2$ denote the orbifold Chern classes and $\alpha_k$, $\beta_k$ are just real numbers depending on $k$.
 
 Using the semi-stability of $T_\mathcal{X}$ with respect to $K_\mathcal{X}$, which is also true for orbifolds (\cite{Kobayashi} and \cite{TY}), one can control the higher cohomology as in \cite{GG80}. In particular, one obtains for an orbifold surface of general type \cite{De95}
\begin{equation}\label{2jet}
h^0(\mathcal{X}, E_{2,m})\geq \frac{m^4}{648}(13c_1^2-9c_2)+O(m^3).
\end{equation}

In the case of smooth hypersurfaces of degree $d$ of $\bP^3$, one obtains that $$h^0(X, E_{2,m}) \geq C.m^4$$ for $d \geq 15$ (and $C>0$). In general, the bigger $c_1^2 /c_2$ is, the easier it is to obtain global jet differentials. This is another illustration that Horikawa surfaces are the most difficult cases to deal with. Therefore, for smooth surfaces in $\bP^3$ of degree $d=5$, one has to take $k$ much larger. In contrast, we will provide below an example of an orbifold surface of $\bP^3$ of degree $5$ where $h^0(\mathcal{X}, E_{2,m})\geq C.m^4$ (for $C>0$).

As in the case of manifolds, one obtains as a consequence of Ahlfors-Schwarz lemma:
\begin{thm}[\cite{McQ1}, \cite{Rousseau2}]\label{vanish}
Let $\pi: \mathcal{X} \to X$ be a projective orbifold and an ample line bundle $A$ on $X$ such that  $H^0(\mathcal{X}_k, \mathcal{O}_{\tilde{\mathcal{X}_k}}(m)\otimes \pi_{0,k}^*A^{-1})\simeq H^0(\mathcal{X}, E_{k,m} \otimes A^{-1})$ has a non-zero section $P$. Then every orbifold entire curve $f: \bC \to \mathcal{X}$ must satisfy the algebraic differential equation $P(f)=0$.
\end{thm}

The generalization of jet differentials to orbifolds can be applied to classical problems which are not yet solved by Nevanlinna theory. Among these problems, we suggest the following one. Consider $f: \bC \to \bP^n$ a holomorphic curve which ramifies with multiplicity divisible by $m_i$ over $H_1,\dots,H_q$, hypersurfaces of degrees $d_i$ of $\bP^n$ in general position. Then it defines an orbifold entire curve $f:\bC \to (\bP^n, \Delta)$ where $\Delta:= \sum_i \left(1-\frac{1}{m_i} \right) H_i.$ The condition $$K_{\bP^n}+\Delta >0,$$ is equivalent to 
$$\sum_i \left(1-\frac{1}{m_i} \right)d_i > n+1.$$ In this setting, Conjecture \ref{OGG} takes the following form

\begin{conjecture}\label{Cartan}
Let $H_1,\dots,H_q$ be hypersurfaces of degrees $d_i$ of $\bP^n$ in general position. Assume that $f: \bC \to \bP^n$ is a holomorphic curve which ramifies over $H_i$ with multiplicity divisible by $m_i$ and that
$$\sum_i \left(1-\frac{1}{m_i} \right)d_i > n+1,$$
then $f$ is algebraically degenerate i.e. its image is contained in a proper algebraic hypersurface.
\end{conjecture}

\subsection{Horikawa orbifolds}
Let us recall that for minimal surfaces of general type, we have the inequalities
$5c_1^2+36 \geq c_2$ if $c_1^2$ is even,
$5c_1^2+ 30 \geq c_2 $ if $c_1^2$ is odd.
Surfaces realizing the equalities are called Horikawa surfaces. Horikawa gave a classification of these surfaces \cite{Hor76}:
\begin{thm}
\label{HorikawaClassif}
Let $Z$ be a Horikawa surface with $c_1^2$ even, then there is a birational map $Z \to X'$ where $X'$ is a surface branched along a curve with $ADE$ singularities, more precisely $X'$ is either
\begin{enumerate}
\item a double covering of $\bP^2$ branched along an octic ($\chi=4$),
\item a double covering of $\bP^2$ branched along a curve of degree $10$ ($\chi=7$),
\item a double covering of $\bF_N$ branched along a curve of type $(6,2a)$, ($2a\geq -N$) ($\chi=3N+2a-1$).
\end{enumerate}
\end{thm}

In the sequel, we will say that a Horikawa surface is of type $(i)$ according to the above classification.

It is therefore natural to associate orbifolds to Horikawa surfaces.
\begin{defn}
We say that an orbifold $\mathcal{X}=(X,\Delta)$ is Horikawa if there is a Horikawa surface $Z$ with a birational map $Z \to X'$ where $X' \to X$ is a branched covering with ramification divisor $\Delta$.
\end{defn}

\begin{example}
From Horikawa's theorem $(\bP^2, \left(1-\frac{1}{2} \right)C)$ where $C$ is a curve of degree $8$ or $10$ with at most $ADE$ singularities is a Horikawa orbifold.

It is possible to construct Horikawa surfaces with $c_1 ^2 =5$  as the (resolutions of) $5$-fold covers of $\bP^2$ branched along five lines. More generally, examples of Horikawa orbifolds are provided by $(\bP^2, \left(1-\frac{1}{5} \right)C)$ where $C$ is a curve of degree $5$ with at most nodes.
\end{example}

\section{Algebraic hyperbolicity}

\subsection{Horikawa surfaces with even first Chern number}

Having in mind Lang's conjecture claiming that there are only
a finite number of curves of genus $0$ or $1$ on surfaces of general
type, the aim of this section is to provide a lower bound on the genus
of a curve on some explicit surfaces. 

To study curves in Horikawa surfaces, one strategy is to study orbifold curves in Horikawa orbifolds. Using this philosophy, we will prove

\begin{thm}\label{Horikawa}
\begin{enumerate}
\item Let $X$ be a very generic Horikawa surface of type $(1)$. Then $X$ has no rational curves.
\item Let $X$ be a very generic Horikawa surface of type $(2)$. Then $X$ is algebraically hyperbolic, in particular $X$ has no rational or elliptic curves.
\item Let $X$ be a very generic Horikawa surface of type $(3)$ with $a=3$ and $\chi=3N+5$. Then $X$ has no rational curves.
\end{enumerate}
\end{thm}

This theorem will be a consequence of more general results which we establish now.

\subsection{Covers of the plane.}

Let $\ensuremath{D=\cup D_{k}\subset\bP^{2}}$, where $D_{k}$ is
a very general curve of degree $d_{k}$. For $n>1$ dividing $d=\sum d_k$ , let us consider $p:X\to\bP^{2}$
the $n$-cyclic covering branched along $D$ (for the construction of cyclic cover see \cite{Barth}). 

The  singularities of $D$ are the intersection points of the $D_k$'s. For such a singularity, there exist local coordinates $z_1 ,z_2$ such that $\Delta$ has equation $z_1 z_2=0$ ; the singularity on the cover is therefore $ \{t^n=z_1 z_2 \}$ it is a $A_{n-1}$. The map $p$ is an orbifold covering between $X$ and $(\bP^2,\Delta)$ where $\Delta=\left(1-\frac{1}{n}\right)\sum D_{i}.$

\begin{thm}
\label{thm Borne inf du genre}Let $f:C \to X$ be an orbifold compact curve in $X$ not contained in the branch locus $D$. Then \[
\deg(K_C) \geq(d-\frac{d}{n}-4)\deg C,\]
 where $\deg C$ is the degree of $C$ computed as $\deg p(f(C)).$\\
For $d>4$ and  $(d,n)\not = (5,5), (6,2), (6,3), (8,2)$, the surface  $X$ is algebraically hyperbolic modulo the rational and elliptic curves in the branch locus of $p:X\to \mathbb{P}^2$.
\\ For $(d,n)= (5,5),  (6,3)$ or $ (8,2)$,  $X$ has no rational curves except the rational curves in the branch locus. 
\end{thm}

\begin{rem}
Of course, if $X$ is smooth, every curve $f: C \to X$ is an orbifold curve.
\end{rem}

\begin{rem}
If there are no rational or elliptic curves in  $\left\lceil \Delta\right\rceil$, then there are no rational or elliptic curves in the branch locus.
\end{rem}

\begin{rem}
The Chern numbers of the desingularisation of $X$ are:\[
\begin{array}{cc}
c_{1}^{2}= & n(-3+(1-\frac{1}{n})d)^{2}\\
c_{2}= & 3n+(n-1)(d^{2}-3d).\end{array}\] 
For $(d,n) = (6,2)$, the surface $X$ is $K3$ : such a surface contains an infinite number of elliptic curves \cite{GG80}. The three other cases $(d,n)= (5,5),  (6,3)$ or $ (8,2)$ are Horikawa surfaces. The case $(5,5)$ is a quintic surface. For $(d,n)= (6,3)$, we have $c_1^2=3$ and $5c_1^2+30=c_2=45$. The case $(d,n)= (8,2)$ is a Horikawa surface of type $(1)$. 
\end{rem}

For the proof of Theorem \ref{thm Borne inf du genre}, we use the
main Theorem of Xi Chen in \cite{Chen01} (see also \cite{PR}):

\begin{thm} Let $D=\cup D_{k}\subset \bP^2$ be as above. For all reduced curve $C\subset\bP^{2}$, we have:\begin{equation}
2g(C)-2+i(C,D)\geq(d-4)\deg C\label{eq:1},\end{equation}
where $d=\sum d_{k}$
and $i(C,D)$ is the number of distinct points of $\nu^{*}D$ if $\nu:C'\to C$
is the normalization.
\end{thm}

\begin{proof}
(Of Theorem \ref{thm Borne inf du genre}). Let $f:C \to X$ be an orbifold map. We have the following
commuting diagram: \[
\xymatrix{C \ar[r]^{f}\ar[d]_{h} & X\ar[d]^{p}\\
(C_{1}',\Delta')\ar[r]_{g} & (\bP^{2},\Delta)}
\]
where $C_{1}=p(f(C))$, $C'_{1}$ is the desingularisation of $C_{1}$,
and $g$ is an orbifold morphism. 
Let\[
\begin{array}{ccc}
g^{*}(D_{j}) & = & \sum_{i=1}^{i(C_{1},D)}t_{i,j}p_{i},\\
g^{*}(D) & = & \sum_{i=1}^{i(C_{1},D)}t_{i}p_{i}.\end{array}\]
 Now, let $\tilde{\Delta}=\sum_{i=1}^{i(C_{1},D)}\left(1-\frac{1}{\tilde{m_{i}}}\right)p_{i}$
be the minimal orbifold structure on $C'_{1}$ such that $g:(C_{1}',\tilde{\Delta})\to(\bP^{2},\Delta)$
is an orbifold morphism. The conditions for $g$ to be an orbifold
morphism are $n|\tilde{m_{i}}.t_{i,j}$ for all $j$. Therefore the
minimal orbifold structure is given by \[
\tilde{m_{i}}=\operatorname{lcm}_{j}\left(\frac{n}{\gcd(n,t_{i,j})}\right).\]
As $t_{i}=\sum_{j}t_{ij}$, we have $\mbox{gcd}(n,t_{ij})\leq t_{i}$
for all $i,j$, thus: \[
\frac{n}{t_{i}}\leq\frac{n}{\gcd(n,t_{i,j})}\leq\operatorname{lcm}_{j}\left(\frac{n}{\gcd(n,t_{i,j})}\right)=\tilde{m_{i}}.\]
So, we have: \[
\sum_{i=1}^{i(C_{1},D)}\left(1-\frac{1}{\tilde{m_{i}}}\right)\geq\sum_{i=1}^{i(C_{1},D)}\left(1-\frac{t_{i}}{n}\right)=i(C_{1},D)-\frac{d}{n}\deg C_{1}.\]
Now, inequality \ref{eq:1} of \cite{Chen01} gives: \[
2g(C_{1}')-2+i(C_{1},D)\geq(d-4)\deg C_{1},\]
therefore:
\[
2g(C_{1}')-2+\sum_{i=1}^{i(C,D)}\left(1-\frac{1}{\tilde{m_{i}}}\right)\geq (d-\frac{d}{n}-4)\deg C_{1}.\]
To conclude, we observe that \[
\deg(K_C) \geq2g(C_{1}')-2+\sum_{i=1}^{i(C,D)}\left(1-\frac{1}{\tilde{m_{i}}}\right)\]
because $C$ is a uniformisation of $C'_1$.
\end{proof}

This implies $(1)$ and $(2)$ of Theorem \ref{Horikawa}.

When $n=d$, the surface $X$ is the  degree $d$ hypersurface in $\mathbb{P}^{3}$ defined
by:\[
X=\{x_{4}^{d}-G(x_{1},x_{2},x_{3})=0\},\]
 where $G$ is a homogeneous polynomial in $x_{1},x_{2},x_{3}$ of degree $d$ defining the curve $C \subset \bP^2$. We can apply the above results to $X$:
\begin{cor}
If $C$ is very generic and $d\geq 6$, then the smooth algebraic surface $X$ is algebraically hyperbolic.
If $C$ is very generic and $d=5$, then $X$ contains no curves of geometric
genus $0$.
\end{cor}
This result should be compared to results of Clemens, Ein, Pacienza, Voisin establishing the algebraic hyperbolicity of very general hypersurfaces of $\bP^n$ of large degree (see \cite{DR} for a survey).

Moreover let us mention that, even if here we are interested in surfaces, since inequality \ref{eq:1} of Chen generalizes to higher dimension (see \cite{Chen01} and \cite{PR}), the proof above immediately generalizes to
\begin{thm}
If $D \subset \bP^n$ is a very generic hypersurface of degree $d$ and $d\geq 2n+2$, then the degree $d$ branched cover $X \subset \bP^{n+1}$ ramified over $D$ is algebraically hyperbolic.
\end{thm}

\subsection{Covers of Hirzebruch surfaces $\mathbb{F}_{N}$, Algebraically hyperbolic
Horikawa surfaces.}
\label{CoversofHirzebruchsurfaces}

In order to prove Theorem \ref{Horikawa} $(3)$, we study covers of Hirzebruch surfaces.

For $N\geq0$, let $\bF_{N}:=\bP(\mathcal{O}\oplus\mathcal{O}(N))$
be the Hirzebruch surface and let $T$ and $F$ the divisors on $\bF_{N}$
generating $Pic(\bF_{N})$ with $T^{2}=N$, $T.F=1$ and $F^{2}=0$.
\\
Let $D=\cup D_{k}\subset\bF_{N}$ where each $D_{k}$ is a very
general member of a base point free complete linear series. And let
$D\sim aT+bF$ with $a(N-1)+b\geq0$. 

\begin{thm}
(Chen, \cite{Chen02} Corollary 1.12). We have:
\begin{equation}
2g(C)-2+i(C,D)\geq\min(a-3,b-2)\deg C \label{equation Chen pour les F_N} \end{equation}

for all reduced irreducible curves $C\subset\bF_{N}$ with $C\not\subset D$
where $\deg C=(T+F).C$.
\end{thm}

As before, for $n$ dividing $a$ and $b$, we consider $p:X\to\bF_{N}$ a $n$-cyclic covering branched
along $D$. Then we have

\begin{thm}
\label{Chen applique aux surfaces de Hirze}Let $f:C \to X$ be an orbifold compact
curve in $X$ not contained in the branch locus. Let $c,d$ be the
integers such that $p(f(C))=C_{1}\sim cT+dF$, then:

\[
\deg (K_C) \geq\min(a-3,b-2)(c(N+1)+d)-\frac{1}{n}(bc+ad+acN).\]

 \end{thm}

\begin{proof}
We follow the notations and proof of theorem \ref{thm Borne inf du genre}
: $C'_{1}$ is the desingularisation of $C_{1}$. First we remark
that \[
\deg C_{1}=c(1+N)+d,\, C_{1}.D=bc+ad+acN.\]
 Then, using the inequality \ref{equation Chen pour les F_N}, we
obtain:
\[
2g(C_{1}')-2+\sum_{i=1}^{i(C,D)}\left(1-\frac{1}{\tilde{m_{i}}}\right)\geq\min\left(a-3,\, b-2\right)(\deg C_{1})-\frac{C_{1}.D}{n}.\]
\end{proof}

We wish to apply Theorem \ref{Chen applique aux surfaces de Hirze}
to the Horikawa surfaces (for which $a=6$, $n=2$). To obtain a non-trivial
result, we need $b\geq3$ if $N\geq1$ and $b\geq6$ if $N=0$ (recall
that $a(N-1)+b\geq0$). Recall moreover that $b$ must be divisible
by $2$. The case $b=4$ gives :\[
\deg (K_C)\geq c(N+1)+d-\frac{1}{2}(bc+6d+6cN)=c(-1-3N)-2d,\]
and therefore is not interesting. If $b\geq6$, we have:\[
\deg (K_C)\geq3(c(N+1)+d)-\frac{1}{2}(bc+6d+6cN)=c(3-\frac{b}{2}),\]
thus we obtain $\deg (K_C)\geq0$ for $a=6$, $b=6$. 

Let us suppose that $D\sim6T+6F$ on $\mathbb{F}_{N}$ ($N\geq0$)
is a very general member of a base point free complete linear series.
The double cover of $\mathbb{F}_{N}$ ramified over $D$ is a Horikawa
surface $X$ with $\chi=3N+5$. That implies $(3)$ of Theorem \ref{Horikawa}.

The proof of Theorem \ref{Horikawa} is now complete.

We would like to emphasize here the fact that this result is valid for
a very generic surface in some irreducible component of the moduli space of Horikawa surfaces.

\subsection{Quintic surfaces.}

\label{Quintic}

Our starting point was the fact that we do not know quasi-hyperbolic
quintics. Consider $5$ lines $L_{1},\dots,L_{5}$ in general position on $\mathbb{P}^{2}$. The
degree $5$ ramified cover of $\mathbb{P}^{2}$ branched over
the $L_{i}$'s is a quintic surface $X$ in $\mathbb{P}^{3}$. If
$\ell_{i}=\ell_{i}(x_{1},x_{2},x_{3})$ is an equation of $L_{i}$,
an equation of $X$ is :\[
X=\{x_{4}^{5}-\ell_{1}\ell_{2}\ell_{3}\ell_{4}\ell_{5}=0\}.\]
As the five lines are in general position, their union is a degree $5$ curve $D$ with $10$ nodal singularities. Over such a singularity $p$, there exists local coordinates $x,y$ such that the equation of $D$ is $xy=0$. The equation of the singularity of $X$ over $p$ is then $z^5-xy=0$ and the surface $X$ contains therefore $10$ $A_4$. We consider this surface $X$ as an orbifold. A corollary of Theorem
\ref{thm Borne inf du genre}  is

\begin{cor} $X$ has no orbifold
rational curves except over the $5$ lines of the ramification locus.
\end{cor}

As $X$ is an orbifold of general type, a natural question is
\begin{problem}
Show  that  $X$ has only finitely many (orbifold) elliptic curves.
\end{problem}

A first result toward this is
\begin{thm}
The image in $\bP^2$ of any orbifold elliptic curve $f: C \to X$ is not rational. More precisely, it is mapped in $\bP^2$ to a singular genus one curve which intersects every line $L_i$ with multiplicity divisible by $5$.
\end{thm}

\begin{proof}
As above, we have the following
 diagram: \[
\xymatrix{C\ar[r]^{f}\ar[d]_{h} & X\ar[d]^{p}\\
(C_{1}',\Delta')\ar[r]_{g} & (\bP^{2},\Delta)}
\]
where $C_{1}=p(f(C))$, $C'_{1}$ is the desingularisation of $C_{1}$,
and $g$ is an orbifold morphism. Therefore $(C_{1}',\Delta'=\sum \left(1- \frac{1}{m_i}\right) p_i)$ is an orbifold elliptic curve i.e. $K_{C_{1}'}+\Delta'= 0$. A case-by-case check gives that $(C_{1}',\Delta')$ is isomorphic to one of the following orbifold curve:
\begin{itemize}
\item $(E,\emptyset)$ where $E$ is an elliptic curve
\item $(\bP^1, \left(1-\frac{1}{m_1} \right)\{0\}+\left(1-\frac{1}{m_2} \right)\{1\}+\left(1-\frac{1}{m_3} \right)\{\infty\}$ where $(m_1, m_2, m_3)$ is either $(2, 3, 6), (2, 4, 4), (3, 3, 3)$
\item $(\bP^1, \left(1-\frac{1}{2} \right)\{0\}+\left(1-\frac{1}{2} \right)\{1\}+\left(1-\frac{1}{2} \right)\{p_3\}+\left(1-\frac{1}{2} \right)\{\infty\}$ with $p_3 \in \bC \setminus \{0, 1\}$.
\end{itemize}
Let\[
\begin{array}{ccc}
g^{*}(L_{j}) & = & \sum_{i=1}^{i(C_{1},L)}t_{i,j}p_{i}.\\
g^{*}(L) & = & \sum_{i=1}^{i(C_{1},L)}t_{i}p_{i}.\end{array}\]

The conditions for $g$ to be an orbifold
morphism are $5|m_{i}.t_{i,j}$ for all $j$. As $\gcd(5, m_i)=1$ by the preceding case-by-case analysis, we conclude that $5|t_{i,j}.$ Therefore $g: C_1' \to (\bP^{2},\Delta)$ is an orbifold morphism. 
From theorem \ref{thm Borne inf du genre}, we deduce that this excludes the case $C_1'= \bP^1$. Therefore in the above list, only the first case one can occur. So $C_1$ is a genus one curve. The condition $5|t_{i,j}$ implies that $\deg(C_1)= \sum_j t_{ij}$ is divisible by $5$ and $C_1$ cannot be a smooth elliptic curve.
\end{proof}

Therefore we have reduced the problem of counting orbifold elliptic curves in these quintics to counting genus one curves in $\bP^2$ which intersects every line $L_i$ with multiplicity divisible by $5$.

\begin{rem}
The following computation shows that it is natural to believe that there are finitely many of these curves.

 The Severi variety $V^{d,\delta}$ of plane curves of degree $d$ and $\delta$ nodes has dimension $3d+g-1.$ Let  $C \in V^{d,\delta}$ be a curve which intersects $L_i$ in $\sum_j \beta_j$ points with multiplicity $5j$. The total number of conditions is $$D:=\sum_{i=1}^5\sum_{j=1}^{l_i} \beta_j (5j-1),$$
with $\sum _j 5j\beta_j=d$.
Then $D=\sum_i (d-\sum_j \beta_j),$ et $\sum \beta_j \leq \frac{d}{5}.$
Therefore $D \geq \sum_i d(1-\frac{1}{5})=4d.$
Thus for $g \leq 1$ the number of conditions is greater than the number of parameters.
\end{rem}

\section{Analytic hyperbolicity}
\subsection{Nevanlinna theory}
Nevanlinna theory can be used to study entire curves in ramified covers. We will briefly recall in this context the truncated defect relation of Cartan following
the notations of \cite{Kobayashi98} to which we refer for details.

Let $f: \bC \to \bP^n$ be a linearly-non degenerate entire curve and $D_1,\dots, D_q$ be $q$ hyperplanes in general position. Let us denote as usual $N^{[n]}(f,r,D_{i})$ and $T(f,r,D_{i})$, the truncated counting function and the characteristic function.

The defect is defined by
 \[
\delta^{[n]}(f,D_{i}):=\liminf_{r->\infty}\left(1-\frac{N^{[n]}(f,r,D_{i})}{T(f,r,D_{i})}\right).\]

Then we have the truncated defect relation of Cartan:

\begin{thm}
$\sum\delta^{[n]}(f,D_{i})\leq n+1.$
\end{thm}

We can apply this to ramified covers of the plane. Let $X \subset \bP^3$ be the degree $d$ cover of $\bP^2$ ramified over $d$ lines $D_i$.

\begin{thm}
If $d\geq6$ then $X$ is a quasi-hyperbolic orbifold.
 \end{thm}

\begin{proof}
Composing with the projection, we obtain an orbifold map $g:\bC\to(\bP^{2},\Delta)$.
If $g$ is linearly degenerate then from Theorem \ref{thm Borne inf du genre},
its image lies in the branch locus. Let us suppose that it is linearly
non-degenerate. By the First Main Theorem of Nevanlinna theory, we
have \[
\frac{d}{2}N^{[2]}(f,r,D_{i})\leq N(f,r,D_{i})\leq T(f,r,D_{i}).\]
 Therefore \[
\delta^{[2]}(f,D_{i})\geq1-\frac{2}{d},\]
 which implies \[
\sum\delta^{[2]}(f,D_{i})\geq d-2,\]
 contradicting Cartan's relation if $d\geq6.$
\end{proof}

Unfortunately we see in the previous proof that we cannot say anything for $d=5$, the case of the quintics we studied.
An interesting question is to prove that we have again algebraic degeneracy
of entire curves in this case. Note that, according to Conjecture \ref{Cartan}, this should be true since
$K_{\bP^{2}}+\Delta=\mathcal{O}(1).$ Surprisingly, it seems that this problem is still open.
\begin{problem}\label{Bloch}
Prove that any entire curve $f: \bC \to \bP^2$ intersecting $5$ lines in general position with multiplicity at least $5$ is algebraically degenerate.
\end{problem}

Nevertheless, we can use Nevanlinna theory to construct some hyperbolic orbifold Horikawa surfaces.

\begin{thm}
Let $a \geq 3$ be an integer. Then there exists a quasi-hyperbolic orbifold Horikawa surface whose minimal resolution has $\chi=2a-1$.
\end{thm}

\begin{proof}
Let $\mathbb{F}_{0}=\mathbb{P}^{1}\times\mathbb{P}^{1}$. Consider the divisor 
$$\Delta= \left(1-\frac{1}{2} \right) \sum_{i=1}^6 G_i + \left(1-\frac{1}{2} \right) \sum_{i=1}^{2a} H_i,$$
where the $G_i$ are fibers of the first projection $p_1$, and the $H_i$ are fibers of the second projection $p_2$. Then $(\mathbb{F}_{0}, \Delta)$ is an Horikawa orbifold. The corresponding Horikawa surface has $\chi=2a-1$.

Let $f: \bC \to (\mathbb{F}_{0}, \Delta)$ be an orbifold entire curve. Then $h:=p_1\circ f: \bC \to \bP^1$ is an entire curve with multiplicity at least $2$ over the $6$ points $a_i:=p_1(G_i).$ 
Then $$\sum\delta^{[1]}(h,a_{i})\geq 6 \left(1-\frac{1}{2} \right).$$

The truncated defect relation (i.e. Nevanlinna Second Main Theorem in this case), which gives the upper bound $2$ for the sum of the defects, implies that $p_1\circ f$ is constant since . Therefore
 $f$ has its image contained in a fiber of $p_2$ and has multiplicity at least $2$ over $2a$ points of this fiber. This implies again that $f$ is constant since $2a \geq 6$. 
\end{proof}

We can apply this kind of construction to numerical quintics.
A quintic surface has $c_1^2=5$ and $c_2=55$. Horikawa \cite{H} proved that a surface with $c_1^2=5$ and $c_2=55$, called a numerical quintic, is either a quintic or a double cover.

Let us consider surfaces of the second type with their attached natural geometric orbifolds $(X, \Delta)$.

\begin{thm}
There exists a numerical quintic whose geometric orbifold is quasi-hyperbolic.
\end{thm}

\begin{proof}
Let $\mathbb{F}_{0}=\mathbb{P}^{1}\times\mathbb{P}^{1}$. Consider the divisor 
$$\Delta= \left(1-\frac{1}{2} \right) \sum_{i=1}^5 G_i + \left(1-\frac{1}{2} \right) (H_1+H_2+C+D_1+D_2),$$
where the $G_i$ are fibers of the first projection $p_1$, the $H_i$ are fibers of the second projection $p_2$, $C$ is a curve of type $(2,1)$, and the $D_i$ are curves of type $(1,1)$ such that:
\begin{enumerate}
\item $G_1$, $H_1$, $C$ and $D_1$ meet in a point $a$.
\item $G_1$, $H_2$, $C$ and $D_2$ meet in a point $b$.
\end{enumerate}
Then $\left\lceil \Delta\right\rceil $ is a curve of type $(6,8)$ with two quadruple points lying on a single fiber of $\mathbb{F}_{0}$.

$(\mathbb{F}_{0}, \Delta)$ is a numerical quintic orbifold.

Let $f: \bC \to (\mathbb{F}_{0}, \Delta)$ be an orbifold entire curve. Then $p_1\circ f: \bC \to \bP^1$ is an entire curve with multiplicity at least $2$ over $5$ points. The truncated defect relation (i.e. Nevanlinna Second Main Theorem in this case) implies that $p_1\circ f$ is constant. Therefore
 $f$ has its image contained in a fiber of $p_2$.
 For a generic fiber i.e. not $G_1$, $f$ has multiplicity at least $2$ over at least $5$ points of this fiber. This implies again that $f$ is constant. Therefore the image of $f$ is contained in $G_1$.
\end{proof}

One can remark that Cartan's theorem unfortunately says nothing in the case of Horikawa surfaces which are degree $2$ covers of the plane $\bP^2$.

\subsection{Jet differentials}
As we have just seen, applications of Nevanlinna theory are quite limited. This is one motivation to develop the tool of orbifold jet differentials.

Let us illustrate it in the case of quintics $X$ that are $\mathbb{Z}/5\mathbb{Z}$-covers of the plane ramified over $5$ lines in general position.
First one could think of using jets of order $1$ to apply Theorem \ref{segre}.

As already mentioned in paragraph \ref{Quintic}, the surface $X$ has a natural structure of orbifold $\mathcal{X}=(X,0)$. 
\begin{prop}
The Chern numbers of $\mathcal{X}$ are:\[
\begin{array}{cc}
c_{1}^{2}(\mathcal{X})= & 5\\
c_{2}(\mathcal{X})= & 7,\end{array}\]
\end{prop}
\begin{proof}
  By paragraph \ref{Quintic}, $X$ contains $10A_4$ singularities. Let us denote by $X'\rightarrow X$ the minimal resolution of these singularities.
The Chern numbers of $X'$ are equals to the Chern numbers of a smooth quintic
(Brieskorn resolution Theorem \cite{Brieskorn}) i.e. $K_{X'}^{2}=5,\, e(X')=55$. We
have $K_{\mathcal{X}}^{2}=K_{X}^{2}=5$ and $c_{2}(\mathcal{X})=e(X')-10\cdot(5-\frac{1}{5})=7$ (see Proposition \ref{ChernNumberOrbi}).
\end{proof}

Therefore the orbifold second Segre number is negative and we cannot apply Theorem \ref{segre}.

Turning to jet differentials of order $2$, one sees that 
$$13c_{1}^{2}(\mathcal{X})-9c_{2}(\mathcal{X})=2>0.$$ 
Therefore by inequality (\ref{2jet}):
\begin{cor}
We have: \[
h^{0}(X,E_{2,m}\otimes L^{-1})\geq2.\frac{m^{4}}{648}+O(m^{3}),\]
for $L$ an ample line bundle on $X$. As a consequence (see (\ref{2jet}) and Theorem \ref{vanish}), any entire orbifold curve $f: \bC \to X$ satisfies a differential equation of order $2$.
\end{cor}

There are many known quintic surfaces with quotient singularities (see
e.g. \cite{Schuett}), but as far as we know, the surfaces that are $\mathbb{Z}/5\mathbb{Z}$-covers of the plane ramified over $5$ lines in general position are the only one such that the associated orbifold satisfies $13c_{1}^{2}(\mathcal{X})-9c_{2}(\mathcal{X})>0$.

This result can also be seen as a first step toward problem \ref{Bloch}.

\section{Persson-Horikawa surfaces }

Let us recall some  notations. For $N$ a positive integer, the $N^{th}$ Hirzebruch surface is  $\mathbb{F}_{N}=\mathbb{P}(\mathcal{O}_{\mathbb{P}^1} \oplus \mathcal{O}_{\mathbb{P}^1} (N))$. We denote by $F$ denote a fiber of the $\mathbb{P}^{1}$-bundle
$\mathbb{F}_{N}\rightarrow\mathbb{P}^{1}$ and by $T$ a section s.t. $T^{2}=N$. We denote by $(a,b)$ a divisor  equivalent to $aT+bF$ in the Néron-Severi group.

Recall (see Theorem \ref{HorikawaClassif}) that a Horikawa surface is the double cover of the plane branched
along an octic (case $\chi=4$), or a curve of degree $10$ (case
$\chi=7$) or a double cover of $\mathbb{F}_{N}$ branched over a
curve $D$ of type $(6,2a)$, with $2a\geq-N$.  In this section we will exhibit quasi-hyperbolic Horikawa orbifolds of each type.

\subsection{Horikawa surfaces with $\chi=2k-1$.}

Let us prove the following Proposition:

\begin{prop}
\label{pro:Horikawa 2k-1}For any integer $k>2$, there exists a quasi-hyperbolic
Horikawa orbifold $\mathcal{X}$ whose minimal resolution has Euler characteristic equals to $\chi=2k-1$.
\end{prop}

\begin{proof}
 The divisor $-2(F+T)$ is  the canonical divisor on  Hirzebruch surface  $\mathbb{F}_{0}=\mathbb{P}^{1}\times\mathbb{P}^{1}$. Following Persson's construction in \cite{Persson} Lemma 4.5, there exists on $\mathbb{F}_{0}$ a $\mathbb{Q}$-divisor: 
\[
\Delta=(1-\frac{1}{k})(F_{1}+F_{2})+(1-\frac{1}{2})(C_{1}+C_{2}+E_{0}+E_{1}+E_{2}+E_{3})\]
 such that the orbifold $(\mathbb{F}_{0},\Delta)$ is uniformizable for all $k\geq2$ and the desingularisation of the uniformization
is a Horikawa surface with Chern invariants $\chi=2k-1$ and $c_{1}^{2}=2\chi-6$. 
The curves $C_{1},C_{2}$ are $(1,1)$-curves meeting with multiplicity 2 into one point. 
The curves $F_{1},F_{2}$ are fibers of the first projection $\mathbb{P}^{1}\times\mathbb{P}^{1}\rightarrow\mathbb{P}^{1}$,
the curves $E_{i}$ are sections of this fibration. The singularities of $\left\lceil \Delta\right\rceil $
are $6A_{1}+4D_{4}+D_{6}$.


 
  \begin{figure}[h]
        \begin{center}
           \begin{overpic}[scale=1]{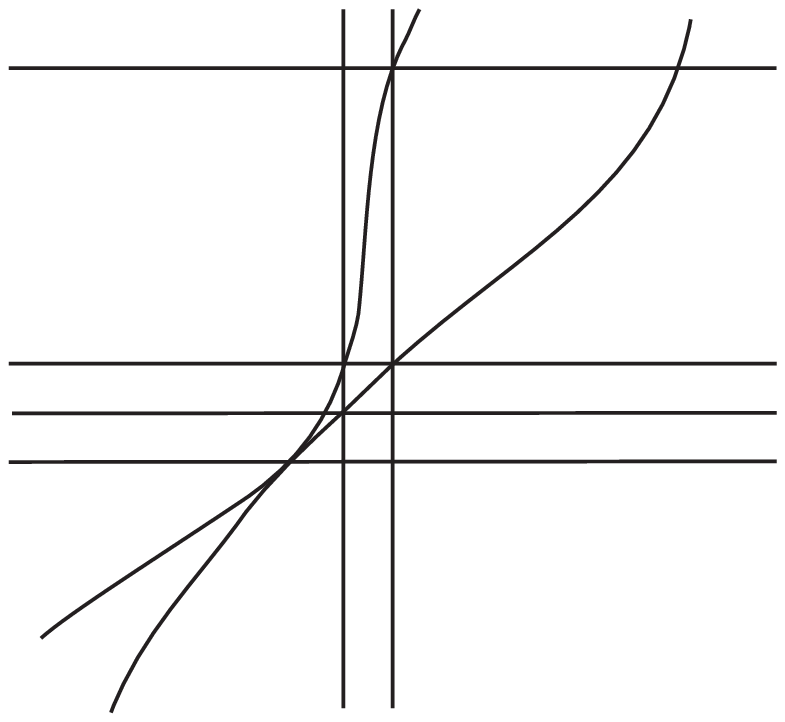}
            \put(-4, 40){$E_3$}
             \put(-4, 47){$E_2$}
             \put(-4, 34){$E_0$}
              \put(-4, 84){$E_1$}
              \put(38,2){$F_1$}
              \put(51,2){$F_2$}
              \put(1,12){$C_1$}
              \put(16,2){$C_2$}
                \end{overpic}
        \end{center}
    \end{figure}

Let us use the results of subsection \ref{SubsectionChernClasses}. Two singularities of type $A_{1}$ are on branches with 
multiplicities $m_{i}$ equal to $2$ (thus $\beta(p)=4$), $4A_{1}$
are on branches with multiplicities $2$ and $k$ (with $\beta(p)=2k$),
the branches of the $4D_{4}$ have the same multiplicities : $2,2,k$
(giving $\beta(p)=4k^{2}$) and the branches of the $D_{6}$ point
have multiplicities $2$, thus $\beta(p)=32$ by Lemma \ref{Beta}. 

Let us compute the orbifold Chern numbers of the associated orbifold
$\mathcal{X}$. We have: $$(K+\Delta)^2=(T+(1-\frac{1}{k})F)^2=T^2+2(1-\frac{1}{k})TF+(1-\frac{1}{k})^2 F^2=2(1-\frac{2}{k})$$ because $F^2=0,$ $FT=1$, $T^2=N=0$ (see section \ref{CoversofHirzebruchsurfaces}).
For the second Chern number of $\mathcal{X}$,  (\ref{Chernnumb}) gives:\[
\begin{array}{c}
c_{2}(\mathcal{X})=4-2(1-\frac{1}{k})(2-4)-\frac{1}{2}(3(2-3)+(2-2)+2(2-4))\\
-(2(1-\frac{1}{4})+4(1-\frac{1}{2k})+4(1-\frac{1}{4k^{2}})+(1-\frac{1}{32}))\end{array}\]
and \[
c_{2}(\mathcal{X})=\frac{33}{32}-\frac{2}{k}+\frac{1}{k^{2}}.\]
Therefore :
\[c_{1}(\mathcal{X})^{2}-c_{2}(\mathcal{X})=\frac{31}{32}-\frac{2}{k}-\frac{1}{k^{2}},\]
and the orbifold $\mathcal{X}$ is quasi-hyperbolic for $k>2$ by Theorem \ref{segre}. 
\end{proof}

\begin{rem}
Let $X \to \mathbb{F}_0$ be the $\mathbb{Z}/2\mathbb{Z} \times \mathbb{Z}/k\mathbb{Z}$-cover of $F_0$  branched over $\Delta$.
The desingularisation of $X$ is a Horikawa surface. The surface $X$ has only $ADE$ singularities and has a natural structure of orbifold $\mathcal{X}'$. The map $\mathcal{X}' \to \mathcal{X}$ is an orbifold covering, in particular $c_{1}(\mathcal{X}')^{2}=2kc_{1}(\mathcal{X})^{2}$ and $c_{2}(\mathcal{X}')=2kc_{2}(\mathcal{X})$.
\end{rem}

\subsubsection{Generalization : a family of quasi-hyperbolic Horikawa orbifold with
$\chi=2k-1$, $k\geq4$.}

Let us consider the configuration of curves  on $\mathbb{F}_0$ in figure \ref{fig:fig01}.

 \begin{figure}
        \begin{center}
           \begin{overpic}[scale=1]{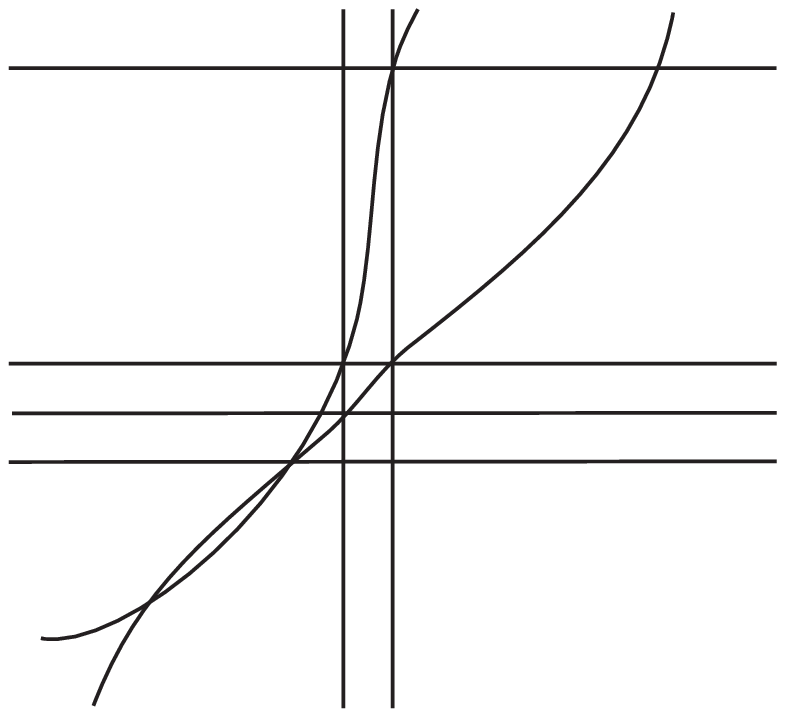}
            \put(-4, 40){$E_3$}
             \put(-4, 47){$E_2$}
             \put(-4, 34){$E_0$}
              \put(-4, 84){$E_1$}
              \put(38,2){$F_1$}
              \put(51,2){$F_2$}
              \put(2,12){$C_1$}
              \put(16,2){$C_2$}
                \end{overpic}
        \end{center}
        \caption{The divisor $\Delta$ on $\mathbb{F}_0$.}
        \label{fig:fig01}
    \end{figure}

The  two $(1,1)$-curves $C_{1},C_{2}$ in $\mathbb{F}_0$ are in general position. 
The singularities of $\Delta$ are $7A_{1}+5D_{4}$. One $D_{4}$
with multiplicities $(2,2,2)$, $4D_{4}$ with multiplicities $(2,2,k)$,
$4A_{1}$ with $(2,k)$ and $3A_{1}$ with $(2,2)$.

Let $\mathcal{X}$ be the associated orbifold. By subsection \ref{SubsectionChernClasses}, its Chern numbers are:
$c_{1}(\mathcal{X})^{2}=2(1-\frac{2}{k})$ and\[
\begin{array}{c}
c_{2}(\mathcal{X})=4-2(1-\frac{1}{k})(2-4)-\frac{1}{2}(3(2-3)+(2-2)+2(2-5))\\
-(3(1-\frac{1}{4})+4(1-\frac{1}{2k})+4(1-\frac{1}{4k^{2}})+(1-\frac{1}{16}))\end{array}\]
thus $c_{2}(\mathcal{X})=\frac{21}{16}-\frac{2}{k}+\frac{2}{k^{2}}$
and\[
c_{1}(\mathcal{X})^{2}-c_{2}(\mathcal{X})=\frac{11}{16}-\frac{2}{k}-\frac{2}{k^{2}}.\]
It is positive if $k\geq4$ and then the orbifold $\mathcal{X}$ is quasi-hyperbolic by Theorem \ref{segre}.

\subsection{Horikawa surfaces with $\chi=4k-1$.}

Let us give another construction of quasi-hyperbolic Horikawa surfaces:

\begin{prop}
For every integer $k>1$, there exists a quasi-hyperbolic orbifold 
Horikawa surface  $\cX$ with \[
\begin{array}{cc}
c_{1}^{2}(\mathcal{X})= & 4(1-\frac{1}{k})\\
c_{2}(\mathcal{X})= & \frac{17}{12}-\frac{2}{k}-\frac{1}{k^{2}}.\end{array}\]
The desingularisation of the $\mathbb{Z}/2\mathbb{Z} \times \mathbb{Z}/k\mathbb{Z}$-cover of the subjacent space to
$\mathcal{X}$ is a Horikawa surface with $\chi=4k-1$.\end{prop}
\begin{proof}

The canonical divisor on $\mathbb{F}_{2}$ is  $K_{\mathbb{F}_{2}}=-2T$.
 Let $k>1$ be an integer. Persson (\cite{Persson} Proposition 4.7) constructed on $\mathbb{F}_{2}$ a divisor :\[
\Delta=(1-\frac{1}{k})(F_{1}+F_{2})+(1-\frac{1}{2})(A+B+C)\]
whose irreducible components $F_{i},A,B,C$ have the following configuration
: $F_{1}$ and $F_{2}$ are fibers, $C$ has type $(1,-2)$, $A$
has type $(2,0)$, $B$ has type $(3,0)$. 

\begin{figure}[h]
        \begin{center}
           \begin{overpic}[scale=1]{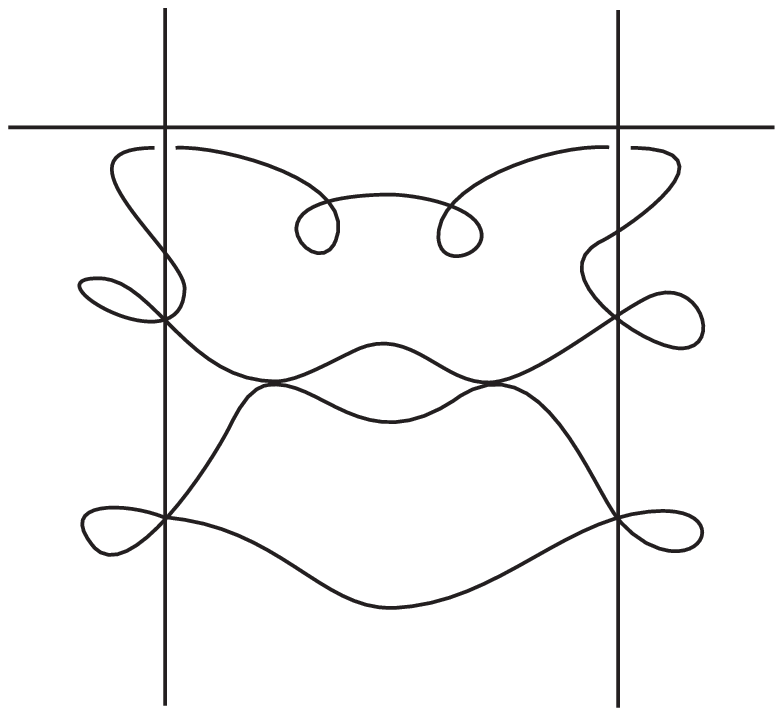}
             \put(48, 49){$B$}
             \put(48, 32){$A$}
              \put(2, 76){$C$}
              \put(16,2){$F_1$}
              \put(80,2){$F_2$}
                \end{overpic}
        \end{center}
         \caption{The divisor $\Delta$ on $\mathbb{F}_0$.}
                 \label{fig:fig02}
    \end{figure}

We have $e(A)=e(B)=0$, $CA=CB=0$ and  $F_{i}A=2,\, F_{i}B=3$. The common points
to $A$ and $B$ are $2A_{11}$ singularities of $\left\lceil \Delta\right\rceil $.\\
The singularities on $\Delta$ are $2A_{1}+2A_{11}$ with
multiplicities $(2,2)$ , $4A_{1}$ with multiplicities $(2,k)$,
and $4D_{4}$ with multiplicities $(2,2,k)$ (see  figure \ref{fig:fig02}).\\
Let us compute the orbifold Chern numbers of $\mathcal{X}=(\mathbb{F}_{2},\Delta)$:\[
c_{1}^{2}(\mathcal{X})=(K_{\mathbb{F}_{2}}+\Delta)^{2}=\left((1-\frac{1}{k})(0,2)+(1-\frac{1}{2})(6,-2)+(-2,0)\right)^{2}=4(1-\frac{1}{k}),\]
and \[
\begin{array}{c}
c_{2}(\mathcal{X})=4-2(1-\frac{1}{k})(2-4)-\frac{1}{2}((0-4)+(0-8)+0)\\
-(2(1-\frac{1}{4})+2(1-\frac{1}{24})+4(1-\frac{1}{2k})+4(1-\frac{1}{4k^{2}})).\end{array}\]
Therefore : $c_{2}(\mathcal{X})=\frac{31}{12}-\frac{2}{k}+\frac{1}{k^{2}}$
and \[
c_{1}(\mathcal{X})^{2}-c_{2}(\mathcal{X})=\frac{17}{12}-\frac{2}{k}-\frac{1}{k^{2}}\]
is positive for $k>1$ and  by Theorem \ref{segre}, $\mathcal{X}$ is quasi-hyperbolic. 
\end{proof}

\subsection{Examples of quasi-hyperbolic orbifold double octic.}

\subsubsection{First construction.}

The Horikawa surfaces with $\chi=4$ are double cover of the plane ramified over an octic with at most $ADE$ singularities.

\begin{prop}
\label{pro:Horikawa double octic}There exists a quasi-hyperbolic orbifold Horikawa
surface whose minimal resolution has $\chi=4$.
\end{prop}

\begin{proof}
Let $\mathcal{X}$ be the orbifold whose subjacent variety is $\mathbb{P}^{2}$
and with $\Delta=\frac{1}{2}(Q+L_{1}+L_{2}+L_{3}+L_{4})$ where $Q$
is the Steiner quartic (the unique quartic curve with $3$ cusps, see \cite{Persson}),
$L_{1},L_{2},L_{3}$ are the tangents to the 3 cusps and $L_{4}$ is
the unique bitangent of $Q$ (see figure  \ref{fig:fig03}).

\begin{figure}[h]
        \begin{center}
           \begin{overpic}[scale=1]{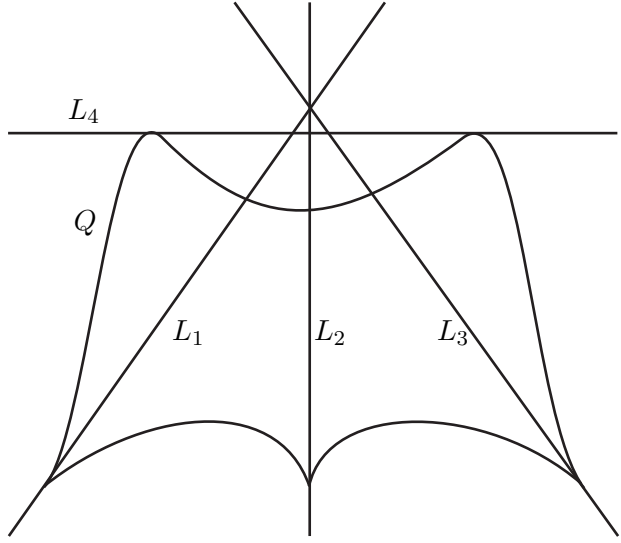}
             \put(11,50){$Q$}
             \put(50, 32){$L_2$}
             \put(10, 68){$L_4$} 
              \put(27,32){$L_1$}
              \put(70,32){$L_3$}
                \end{overpic}
        \end{center}
                 \caption{The Steiner quartic, its bitangent and the 3 tangents to its cusps.}
        \label{fig:fig03}
    \end{figure}

The $2$ points where $L_4$ is tangent to $Q$ are $A_3$ singularities of $\left\lceil \Delta\right\rceil $. The singularity of $\left\lceil \Delta\right\rceil $ at the cusp $Q$ have type $E_7$. The three lines $L_1,\dots,L_3$ meet at a unique point, giving a $D_4$ singularity. The singularities of the curve $\left\lceil \Delta\right\rceil $
are $6A_{1}+2A_{3}+D_{4}+3E_{7}$. We have:
\[ c_{1}^{2}(\mathcal{X})=(-3+(1-\frac{1}{2})8)^{2}=1.\]
Moreover: 
 \[c_{2}(\mathcal{X})=3-\frac{1}{2}(e(Q\setminus S)+\sum_{i}e(L_{i}\setminus S))-\sum_{p\in S}1-\frac{1}{\beta(p)},\] 
where $S$ is the set of singularities of $\left\lceil \Delta\right\rceil $. Therefore:\[
c_{2}(\mathcal{X})=3-\frac{1}{2}(2-8+3(2-4)+(2-5))-12+\frac{6}{4}+\frac{2}{8}+\frac{1}{16}+\frac{3}{96}=\frac{11}{32},\]
and $c_{1}^{2}(\mathcal{X})-c_{2}(\mathcal{X})=\frac{21}{32}>0$. By Theorem \ref{segre}, $\mathcal{X}$ is quasi-hyperbolic.
The desingularisation of the degree $2$ cover ramified over $\left\lceil \Delta\right\rceil $
is a Horikawa surface with $\chi=4$.
\end{proof}

\subsubsection{Second construction: a pencil of quasi-hyperbolic Horikawa surfaces.}
\begin{prop}
There exists a pencil of quasi-hyperbolic orbifold Horikawa surfaces whose minimal resolutions have $\chi=4$.
\end{prop}

\begin{proof}
Let us consider the curve in Proposition
\ref{pro:Horikawa double octic}, and replace the line $L_3$ by a generic line $L_3 '$ going through the intersection
point of $L_1$ and $L_2$. 


That gives
a degree $8$ curve $\Delta$ with singularities \[
9A_{1}+2A_{3}+D_{4}+2E_{7}\]
For which $c_{2}(\mathcal{X})=\frac{3}{4}$ and  $c_{1}^{2}(\mathcal{X}) - c_{2}(\mathcal{X})>0$. By Theorem \ref{segre}, $\mathcal{X}$ is quasi-hyperbolic.
\end{proof}

\subsection{A family of quasi-hyperbolic double covers branched over a degree $10$ curve.}

Let us prove:

\begin{prop}
There exists a $4$ dimensional family of quasi-hyperbolic Horikawa
orbifolds whose minimal resolutions have $\chi=7$.
\end{prop}

\begin{proof}
Let  $C$ be the degree
$10$ curve that is the union of the degree $8$ curve in Proposition
\ref{pro:Horikawa double octic} and two lines in general position.
It has singularities:\[
23A_{1}+2A_{3}+d_{4}+3E_{7}.\]
The orbifold Chern classes of $\mathcal{X}=(\mathbb{P}^{2},C)$
are :\[
c_{1}^{2}(\mathcal{X})=(-3+\frac{10}{2})^{2}=4\]
and \[
c_{2}(\mathcal{X})=3-\frac{1}{2}(2-16+3(2-6)+(2-7)+2(2-9))-29+\frac{23}{4}+\frac{2}{8}+\frac{1}{16}+\frac{3}{96}=\frac{83}{32},\]
thus : $c_{1}^{2}(\mathcal{X})-c_{2}(\mathcal{X})>0$ and by Theorem \ref{segre}, $\mathcal{X}$ is quasi-hyperbolic. The moduli of $2$ lines in $\mathbb{P}^{2}$ is $4$ dimensional.
\end{proof}

\newpage
\noindent
Xavier Roulleau,\\ 
{\tt roulleau@math.ist.utl.pt}\\
D\'epartemento de Mathematica,\\
Instituto Superior T\'ecnico,\\ 
Avenida Rovisco Pais\\
1049-001 Lisboa\\
Portugal

\bigskip
\noindent
Erwan Rousseau,\\
{\tt erwan.rousseau@cmi.univ-mrs.fr}\\
Laboratoire d'Analyse, Topologie, Probabilit\'es\\
Universit\'e d'Aix-Marseille et CNRS\\
39, rue Fr\'ed\'eric Joliot-Curie\\
13453 Marseille Cedex 13\\
France

\end{document}